\documentclass[12pt]{iopart}

\usepackage{iopams}
\usepackage{graphicx}
\usepackage{subfigure}
\newtheorem{theorem}{Theorem}[section]
\newtheorem{lemma}{Lemma}[section]

\newtheorem{definition}{Definition}[section]

\newenvironment{proof}{{\noindent\bf Proof.}}{\hfill $\square$\par}
\begin{document}

\title[An inverse random source problem for the time-space fractional diffusion equation]{An inverse random source problem for the time-space fractional diffusion equation driven by fractional Brownian motion}

\author{Daxin Nie \& Weihua Deng}

\address{School of Mathematics and Statistics, Gansu Key Laboratory of Applied Mathematics and Complex Systems, Lanzhou University, Lanzhou 730000, P.R. China}
\ead{dengwh@lzu.edu.cn}
\vspace{10pt}

\begin{abstract}
We study the inverse random source problem for the time-space fractional diffusion equation driven by fractional Brownian motion with Hurst index $H\in(0,1)$.  
With the aid of a novel estimate, by using the operator approach we propose regularity analyses for the direct problem. Then we provide a reconstruction scheme for the source terms $f$ and $g$ up to the sign. Next, combining the properties of Mittag-Leffler function, the complete uniqueness and instability analyses are provided. It's worth mentioning that all the analyses are unified for $H\in(0,1)$.

\end{abstract}

%
\vspace{2pc}
\noindent{\it Keywords}: Stochastic fractional diffusion equation, inverse random source problem, Riemann-Liouville fractional derivative, reconstruction algorithm
%
%
%
%

\section{Introduction}\label{Sec1}
We consider the stochastic time-space fractional diffusion equation with a random source term
\begin{equation}\label{eqretosol}
	\left \{\eqalign{
		&\partial_{t}u(x,t)+\partial^{1-\alpha}_{t} A^{s}u(x,t)=f(x)h(t)+g(x)\dot{W}^{H}(t)\qquad (x,t)\in D\times(0,T],\\
		&u(x,0)=0\qquad x\in D,\\
		&u(x,t)=0\qquad (x,t)\in \partial D\times(0,T],}\right .
\end{equation}
where $D$ is a bounded domain with Lipschitz boundary;  $A^{s}$ with $s\in(0,1)$ is the fractional Laplacian defined by
\begin{equation*}
	A^{s}u=\sum_{k=1}^{\infty}\lambda_{k}^{s}(u,\phi_{k})\phi_{k},
\end{equation*}
in which $\{\lambda_{k},\phi_{k}\}_{k=1}^{\infty}$ are the  nondecreasing eigenvalues and $L^{2}$-norm normalized eigenfunctions of $A=(-\Delta)$ with a zero Dirichlet boundary condition; $f(x),g(x)\in L^{2}(D)$ are deterministic terms with $\|g(x)\|_{L^{2}(D)}\neq0$; $h\in L^{\infty}([0,T])$ and there exists a positve constant $C_{h}$ satisfying $h(t)\geq C_{h}>0$ with $t\in(0,T)$; $\partial^{1-\alpha}_{t}$ with $\alpha\in(0,1)$ means Riemann-Liouville fractional derivative defined by \cite{Podlubny.1999Fde}
\begin{equation*}
	\partial^{1-\alpha}_{t}u=\frac{1}{\Gamma(\alpha)}\frac{\partial}{\partial t}\int_{0}^{t}(t-r)^{\alpha-1}u(r)dr;
\end{equation*}
$W^{H}(t)$ with Hurst index $H\in(0,1)$ is the fractional Brownian motion (fBm) on a complete probability space $(\Omega,\mathbb{F},\mathbb{P})$, in which $\Omega$ is a sample space, $\mathbb{F}$ is a $\sigma$-algebra on $\Omega$, and $\mathbb{P}$ is a probability
measure on the measurable space $(\Omega,\mathbb{F})$.

It is well known that the fractional diffusion equation works well in describing the anomalous diffusion phenomena \cite{Barkai.2001FFPesaa,Metzler.1999ADaRCtTEAFFPEA,Metzler.1998AtiefCtrwafdee}. While the system (\ref{eqretosol}) without source terms governs the probability density function of the subordinate killed Brownian motion \cite{Liu.2021Deng}. More often, the practical physical system has deterministic and/or stochastic sources. The source terms for the system (\ref{eqretosol}) are the deterministic one and the fractional Gaussian noise. In fact, considering the non-ignorable random disturbance, the stochastic fractional partial differential  equations have already been widely concerned by scholars mathematically and numerically \cite{Gunzburger.2018ScrotdfstfPstastwn,Gunzburger.2019Cofesospidedbwn,Sakamoto.2011Ivbvpffdweaatsip,Wu.2020AaotLsfsspdbistwn,Yan.2019OeeffspdewfBm}.



The inverse source problem is an important topic for anomalous dynamics, involving the inverse source problem in time, space, and time-space fractional diffusion equations and so on \cite{Babaei.2019Runbciatfirdcp,Huang.2019Cefttfadeaa,Jin.2021Aippfssar,Kaltenbacher.2019Oaippfafrde,Li.2019Aipitfdewnbc,Qasemi.2019TtfdipstaembaldGm,Thach.2019Ioaispftfdewrn,Tuan.2019Aipfaitfdearmaee,Yan.2019Isdspfatfdebaapa,Zhang.2017Rtptiafde}. As for the inverse random source problems of the fractional diffusion equations, the corresponding discussions seem to be few  \cite{Feng.2020AirspfttfdedbafBm,Liu.2020Rottdstiasfde,Niu.2020Airspiasfde}. In \cite{Niu.2020Airspiasfde}, by the properties and the It\^{o} isometry of Brownian motions, the authors propose the regularity of the fractional diffusion equations driven by Brownian motions and discuss the reconstruction of source terms and its instability. In \cite{Feng.2020AirspfttfdedbafBm}, the authors consider the inverse random source problem of fractional diffusion equation with fBms, to be specific, by transforming the Wiener integral
with respect to fBm into the one of Brownian motion, the authors provide the well-posedness of the direct problem and perform the instability analyses of the inverse problem for $H\in(0,\frac{1}{2})$ and $H\in(\frac{1}{2},1)$, separately.

In this paper,  we first provide a unified discussion on the regularity of direct problem for all $H\in(0,1)$. Then we propose a reconstruction scheme for $f$ and $|g|$ and discuss the uniqueness. By extending the estimates of the stochastic integral with respect to fBms with $H\in(0,1)$ in \cite{Nie.2021428AucaftfdedbfGnwHiHi01}, a unified instability analysis for $H\in(0,1)$ is proposed. Different from the instability discussions in \cite{Feng.2020AirspfttfdedbafBm}, we not only show the instability of recovering $g_{k}^{2}$, but also for $g_{k}g_{l}$.

The rest of the paper is organized as follows. In Section \ref{Sec2}, we provide some facts on fBms and Mittag-Leffler functions. The well-posedness of the direct problem is discussed in Section \ref{Sec3}. Then, we propose the reconstructions of $f$ and $|g|$ and the uniqueness and instability of the reconstructions are established. In Section \ref{Sec4}, we perform some numerical experiments to validate the theoretical results. At last, we conclude the paper with some discussions.

\section{Preliminaries}\label{Sec2}
In this section, we provide some useful lemmas and present the solution of Eq. \eref{eqretosol}. In the following, denote $\mathbb{E}$ as the expectation and ${\rm Cov}(X,Y)=\mathbb{E}[(X-\mathbb{E}(X))(Y-\mathbb{E}(Y))]$ for variables $X$ and $Y$, $\epsilon>0$ is arbitrarily small, and $C$ is a positive constant, whose value may differ at different places.

Let us first recall the definitions of Riemann-Liouville fractional integrals.
\begin{definition}[\cite{Podlubny.1999Fde}]
	The left- and right-sided Riemann-Liouville fractional integrals of order $\alpha$ $(\alpha>0)$ are defined by
	\begin{equation*}
		\eqalign{
			{}_{a}\partial^{-\alpha}_{x}u=\frac{1}{\Gamma(\alpha)}\int_{a}^{x}(x-\xi)^{\alpha-1}u(\xi)d\xi,\\
			{}_{x}\partial^{-\alpha}_{b}u=\frac{1}{\Gamma(\alpha)}\int_{x}^{b}(\xi-x)^{\alpha-1}u(\xi)d\xi
		}
	\end{equation*}
	with $a,b\in\mathbb{R}$.
\end{definition}
And they have the following properties.
\begin{lemma}\label{lemIIiso}
	For $u,v\in L^{2}(\mathbb{R})$, $\nu\in(0,\frac{1}{2})$, and ${\bf supp}~ u,~{\bf supp}~ v\subset [a,b]$ with $a,b\in \mathbb{R}$, we have
	\begin{equation*} \fl
		\int_{a}^{b}{}_{a}\partial^{-\nu}_{x}u~_{x}\partial^{-\nu}_{b}vdx+\int_{a}^{b}{}_{a}\partial^{-\nu}_{x}v~_{x}\partial^{-\nu}_{b}udx=\frac{1}{\Gamma(2\nu)}\int_{a}^{b}\int_{a}^{b} u(\xi)v(\eta)|\eta-\xi|^{2\nu-1}d\eta  d\xi
	\end{equation*}
	and
	\begin{equation*}
		\eqalign{
			\int_{a}^{b}\int_{a}^{b} u(\xi)v(\eta)|\eta-\xi|^{2\nu-1}d\eta  d\xi\leq C\|{}_{a}\partial^{-\nu}_{x}u\|_{L^{2}([a,b])}\|{}_{a}\partial^{-\nu}_{x}v\|_{L^{2}([a,b])},\\
			\int_{a}^{b}\int_{a}^{b} u(\xi)v(\eta)|\eta-\xi|^{2\nu-1}d\eta  d\xi\leq C\|{}_{x}\partial^{-\nu}_{b}u\|_{L^{2}([a,b])}\|{}_{x}\partial^{-\nu}_{b}v\|_{L^{2}([a,b])}.
		}
	\end{equation*}
	
\end{lemma}

\begin{proof}
	For ${\bf supp}~u,~{\bf supp}~ v\subset (a,b)$, simple calculations give
	\begin{equation*}
		\eqalign{
			&\int_{a}^{b}{}_{a}\partial^{-\nu}_{x}u~_{x}\partial^{-\nu}_{b}vdx\\
			=&\frac{1}{(\Gamma(\nu))^{2}}\int_{a}^{b}\int_{a}^{x}(x-\xi)^{\nu-1}u(\xi)d\xi \int_{x}^{b}(\eta-x)^{\nu-1}v(\eta)d\eta dx\\
			=&\frac{1}{(\Gamma(\nu))^{2}}\int_{a}^{b}\int_{a}^{x}\int_{x}^{b}(x-\xi)^{\nu-1} (\eta-x)^{\nu-1}u(\xi)v(\eta)d\eta d\xi  dx\\
			=&\frac{1}{(\Gamma(\nu))^{2}}\int_{a}^{b}\int_{\xi}^{b}\int_{x}^{b}(x-\xi)^{\nu-1} (\eta-x)^{\nu-1}u(\xi)v(\eta)d\eta dx d\xi  \\
			=&\frac{1}{(\Gamma(\nu))^{2}}\int_{a}^{b}\int_{\xi}^{b}\int_{\xi}^{\eta}(x-\xi)^{\nu-1} (\eta-x)^{\nu-1}dx u(\xi)v(\eta)d\eta  d\xi  \\
			=&\frac{1}{\Gamma(2\nu)}\int_{a}^{b}\int_{\xi}^{b} u(\xi)v(\eta)(\eta-\xi)^{2\nu-1}d\eta  d\xi.  \\
		}
	\end{equation*}
	Similarly, one can get
	\begin{equation*}
		\eqalign{
			\int_{a}^{b}{}_{a}\partial^{-\nu}_{x}v~_{x}\partial^{-\nu}_{b}udx&=\frac{1}{\Gamma(2\nu)}\int_{a}^{b}\int_{\xi}^{b} v(\xi)u(\eta)(\eta-\xi)^{2\nu-1}d\eta  d\xi\\
			&=\frac{1}{\Gamma(2\nu)}\int_{a}^{b}\int_{a}^{\eta} v(\xi)u(\eta)(\eta-\xi)^{2\nu-1}  d\xi d\eta.
		}
	\end{equation*}
	Thus
	\begin{equation*} \fl
		\int_{a}^{b}{}_{a}\partial^{-\nu}_{x}u~_{x}\partial^{-\nu}_{b}vdx+\int_{a}^{b}{}_{a}\partial^{-\nu}_{x}v~_{x}\partial^{-\nu}_{b}udx=\frac{1}{\Gamma(2\nu)}\int_{a}^{b}\int_{a}^{b} u(\xi)v(\eta)|\eta-\xi|^{2\nu-1}d\eta  d\xi.
	\end{equation*}
	For ${\rm\bf supp}~u,{\rm\bf supp}~v\subset [a,b]$, and $\nu\in(0,\frac{1}{2})$, Parseval's equality leads to
	\begin{equation}\label{eqIIiso}
		\eqalign{
			&\int_{a}^{b}{}_{a}\partial^{-\nu}_{x}u~{}_{x}\partial^{-\nu}_{b}vdx\\
			=&\int_{\mathbb{R}}{}_{-\infty}\partial^{-\nu}_{x}u~{}_{x}\partial^{-\nu}_{\infty}vdx\\
			=&\cos(\nu\pi)\int_{\mathbb{R}}|\omega|^{-2\nu}\mathcal{F}(u)(\omega)\mathcal{F}(v)(\omega)d\omega\\
			\leq&C\cos(\nu\pi)\left(\int_{\mathbb{R}}|\omega|^{-2\nu}(\mathcal{F}(u)(\omega))^{2}d\omega\right)^{\frac{1}{2}}\left(\int_{\mathbb{R}}|\omega|^{-2\nu}(\mathcal{F}(v)(\omega))^{2}d\omega\right)^{\frac{1}{2}}\\
			\leq&C\left(\int_{a}^{b}{}_{a}\partial^{-\nu}_{x}u~{}_{x}\partial^{-\nu}_{b}udx\right)^{\frac{1}{2}}\left(\int_{a}^{b}{}_{a}\partial^{-\nu}_{x}v~{}_{x}\partial^{-\nu}_{b}vdx\right)^{\frac{1}{2}},
		}
	\end{equation}
	where $\mathcal{F}(u)$ means the Fourier transform of $u$. According to \cite{Nie.2021428AucaftfdedbfGnwHiHi01}, we have for ${\bf supp}~u\subset [a,b]$ and $\nu\in(0,\frac{1}{2})$
	\begin{equation*}
		\eqalign{
			\int_{a}^{b}{}_{a}\partial^{-\nu}_{x}u~{}_{x}\partial^{-\nu}_{b}udx\leq& C\|{}_{a}\partial^{-\nu}_{x}u\|_{L^{2}([a,b])}^{2}.
		}
	\end{equation*}
	Combining \eref{eqIIiso}, we get the first desired result and the second one follows by similar arguments.
\end{proof}

Then, we provide some lemmas about one-dimensional fBm.
\begin{lemma}[\cite{Bardina.2006MfiwHplt12,Cao.2017ASEEwAWaRN,Cao.2018FeafsosdedbfBm}]\label{Lemitoeql05}
	For $H\in(0,1/2)$ and $q_{1}(t),q_{2}(t)\in H^{\frac{1-2H}{2}}_{0}([0,T])$ $($$H^{\frac{1-2H}{2}}_{0}([0,T])=\left \{v\in H^{s}(\mathbb{R}),\mathbf{ supp}~v\in [0,T]\right \}$$)$, we have
	\begin{equation}\label{equisol05}
		 \fl\eqalign{
			&\mathbb{E} \left [\int_{0}^{T} q_{1}(r) d W^{H}(r) \int_{0}^{T} q_{2}(r) d W^{H}(r)\right ]\\
			&\qquad\qquad\qquad=\frac{1}{2}H(1-2H) \int_{\mathbb{R}} \int_{\mathbb{R}} \frac{(q_{1}(r_{1})-q_{1}(r_{2}))(q_{2}(r_{1})-q_{2}(r_{2}))}{|r_{1}-r_{2}|^{2 -2H}} d r_{1} d r_{2},
		}
	\end{equation}
	where $W^H$ means one-dimensional fBm.
\end{lemma}

\begin{lemma}[\cite{Kloeden.1992Nsosde,Mishura.2008ScffBmarp}]\label{Lemitoeqge05}
	For $H\in(1/2,1)$ and $g_{1}(t),g_{2}(t)\in L^2([0,T])$, there holds
	\begin{equation}\label{equisog05}
		\eqalign{
			&\mathbb{E} \left [\int_{0}^{T} q_{1}(r) d W^{H}(r) \int_{0}^{T} q_{2}(r) d W^{H}(r)\right ]\\
			&\qquad\qquad\qquad=H(2 H-1) \int_{0}^{T} \int_{0}^{T} q_{1}(r_{1}) q_{2}(r_{2})|r_{1}-r_{2}|^{2 H-2} d r_{1} d r_{2},
		}
	\end{equation}
	where $W^H$ is one-dimensional fBm.
\end{lemma}

\begin{lemma}\label{eqcorleml}
	Let $q_{1},q_{2}\in L^{2}([0,T])$,	${}_{0}\partial^{\frac{1-2H}{2}}_{t}q_{1},{}_{0}\partial^{\frac{1-2H}{2}}_{t}q_{2}\in L^{2}([0,T])$, and $H\in(0,1)$. Then we have
	\begin{equation*}
		 \fl\eqalign{
			&\mathbb{E} \left (\int_{0}^{T} q_{1}(r) d W^{H}(r) \int_{0}^{T} q_{2}(r) d W^{H}(r)\right )\leq C\left \|{}_{0}\partial^{\frac{1-2H}{2}}_{t}q_{1}\right \|_{L^{2}([0,T])}\left \|{}_{0}\partial^{\frac{1-2H}{2}}_{t}q_{2}\right \|_{L^{2}([0,T])},\\
			&\mathbb{E} \left (\int_{0}^{T} q_{1}(T-r) d W^{H}(r) \int_{0}^{T} q_{2}(T-r) d W^{H}(r)\right )\leq C\left \|{}_{0}\partial^{\frac{1-2H}{2}}_{t}q_{1}\right \|_{L^{2}([0,T])}\left \|{}_{0}\partial^{\frac{1-2H}{2}}_{t}q_{2}\right \|_{L^{2}([0,T])}.
		}
	\end{equation*}
\end{lemma}
\begin{proof}
	When $H\in(\frac{1}{2},1)$, the desired results can be got by Lemmas \ref{lemIIiso} and \ref{Lemitoeqge05} directly. 
	For $H=\frac{1}{2}$, the desired results can be got by It{\^o} isometry and Cauchy-Schwarz inequality.
		When $H\in(0,\frac{1}{2})$, using Lemma \ref{Lemitoeql05},  Parseval's equality and Cauchy-Schwarz inequality, one has
	\begin{equation*}
		\eqalign{
			&\mathbb{E} \left [\int_{0}^{T} q_{1}(r) d W^{H}(r) \int_{0}^{T} q_{2}(r) d W^{H}(r)\right ]\\
			\leq&C\left (\int_{\mathbb{R}} \int_{\mathbb{R}} \frac{(q_{1}(r_{1})-q_{1}(r_{2}))^{2}}{|r_{1}-r_{2}|^{2 -2H}} d r_{1} d r_{2}\right )^{\frac{1}{2}}\left (\int_{\mathbb{R}} \int_{\mathbb{R}} \frac{(q_{2}(r_{1})-q_{2}(r_{2}))^{2}}{|r_{1}-r_{2}|^{2 -2H}} d r_{1} d r_{2}\right )^{\frac{1}{2}}
		}
	\end{equation*}
	Following the facts that the equivalence of $H^{s}(a,b)$ and $H^{s}_{0}(a,b)$ with $s\in (0,\frac{1}{2})$ \cite{Adams.2003Ss,Ervin.2006Vfftsfade,Nie.2021428AucaftfdedbfGnwHiHi01} and
	\begin{equation*}
		C_{1}\|{}_{a}\partial ^{s}_{x}u\|_{L^{2}(a,b)}\leq |u|_{H^{s}(a,b)}\leq C_{2}\|{}_{a}\partial ^{s}_{x}u\|_{L^{2}(a,b)},
	\end{equation*}
	one can get the first desired results. As for the second one, it can be obtained similarly.    
\end{proof}
Next, we give the expression of mild solution. To get the mild solution of Eq. \eref{eqretosol}, we introduce the Mittag-Leffler function \cite{Gorenflo.2014MLFRTaA},
\begin{equation*}
	E_{\alpha,\beta}(z)=\sum_{k=0}^{\infty}\frac{z^{k}}{\Gamma(k\alpha+\beta)}
\end{equation*}
for $z\in \mathbb{C}$. According to \cite{Gorenflo.2014MLFRTaA}, the Laplace transform of $t^{\beta-1}E_{\alpha,\beta}(\pm\lambda t^{\alpha})$ with $\alpha,\beta>0$ can be written as
\begin{equation*}
	\int_{0}^{\infty}e^{-zt}t^{\beta-1}E_{\alpha,\beta}(\pm\lambda t^{\alpha})dt=\frac{z^{\alpha-\beta}}{z^{\alpha}\mp \lambda}.
\end{equation*}
For the Mittag-Leffler function, we have the following lemma.
\begin{lemma}[\cite{Sakamoto.2011Ivbvpffdweaatsip}]\label{lemmtal}
	The function $E_{\alpha,1}(t)$ with $\alpha\in(0,1)$ is completely monotonic, i.e.,
	\begin{equation*}
		(-1)^{n}\frac{d^{n}E_{\alpha,1}(-t)}{dt^{n}}\geq 0,\quad t\geq0,~n\in \mathbb{N}.
	\end{equation*}
\end{lemma}
Introduce 
\begin{equation*}
	E(t)=\frac{1}{2\pi\mathbf{i}}\int_{\Gamma_{\theta,\kappa}}e^{zt}z^{\alpha-1}(z^{\alpha}+A^{s})^{-1}dz,
\end{equation*}
which can also be written as
\begin{equation*}
	E(t)u=\sum_{k=1}^{\infty}E_{\alpha,1}(-\lambda_{k}^{s}t^{\alpha})(u,\phi_{k})\phi_{k}.
\end{equation*}
Let $\Gamma_{\theta,\kappa}$ with $\theta\in(\frac{\pi}{2},\pi)$ and $\kappa>0$ be a contour defined by
\begin{equation*}
	\Gamma_{\theta,\kappa}=\{re^{-\mathbf{i}\theta}:r\geq \kappa\}\cup\{\kappa e^{\mathbf{i}\psi}:|\psi|\leq \theta\}\cup\{re^{-\mathbf{i}\theta}:r\geq \kappa\}.
\end{equation*}
By the interpolation properties \cite{Adams.2003Ss}, it is easy to verify that for $z\in\Sigma_{\theta}=\{z\in\mathbb{C},z\neq 0,|\arg z|\leq \theta\}$ $(\theta\in(\frac{\pi}{2},\pi))$, one has
\begin{equation}\label{eqresoest}
	\|A^{\sigma s}\tilde{E}\|\leq C|z|^{\sigma\alpha-1},
\end{equation}
where $\sigma\in[0,1]$ and $\|\cdot\|$ means the operator norm from $L^{2}(D)$ to $L^{2}(D)$.
Thus by Laplace transform, we can write the solution of Eq. \eref{eqretosol} as
\begin{equation}\label{eqsolrep1}
	u(x,t)=\int_{0}^{t}E(t-r)h(r)drf(x)+\int_{0}^{t}E(t-r)dW^{H}(r)g(x).
\end{equation}

\section{The direct problem}\label{Sec3}
In this section, we consider the well-posedness of Eq. \eref{eqretosol}. Different from the discussions in \cite{Feng.2020AirspfttfdedbafBm,Niu.2020Airspiasfde}, we develop the regularity theory by operator approach (one can refer to \cite{Pruss.2012?Eieaa}) and Lemma \ref{eqcorleml}, which help us give a unified proof for all $H\in(0,1)$ and simplify the proof significantly.
\begin{lemma}
	Let $u$ be the solution of Eq. \eref{eqretosol}, $f(x),g(x)\in L^{2}(D)$, and $h(t)\in L^{\infty}([0,T])$. 
 Then we have
	\begin{equation*}
		\mathbb{E}\|u\|^{2}_{L^{2}(D\times [0,T])}\leq C\|h\|_{L^{\infty}([0,T])}^{2}\|f(x)\|_{L^{2}(D)}^{2}+CT^{2H+1}\|g(x)\|_{L^{2}(D)}^{2}
	\end{equation*}
	and
	\begin{equation*}
		\mathbb{E}\|A^{\sigma s}u(t)\|^{2}_{L^{2}(D)}\leq C\|h\|_{L^{\infty}([0,T])}^{2}\|f(x)\|_{L^{2}(D)}^{2}+Ct^{2H-2\sigma\alpha}\|g(x)\|_{L^{2}(D)}^{2},
	\end{equation*}
	where $\sigma\in[0, \min(1,\frac{H}{\alpha}-\epsilon)]$ with $\epsilon>0$ arbitrarily small.
\end{lemma}
\begin{proof}
	According to \eref{eqsolrep1}, we have
	\begin{equation*}
		\eqalign{
			\mathbb{E}\|u\|^{2}_{L^{2}(D\times [0,T])}&=\mathbb{E}\int_{0}^{T}\|u\|_{L^{2}(D)}^{2}dt\\
			&\leq C\int_{0}^{T}\left \|\int_{0}^{t}E(t-r)h(r)dr\right \|^{2}\|f(x)\|_{L^{2}(D)}^{2}dt\\
			&\quad+C\int_{0}^{T}\mathbb{E}\left \|\int_{0}^{t}E(t-r)dW^{H}(r)\right \|^{2}\|g(x)\|_{L^{2}(D)}^{2}dt\\
			&\leq C\int_{0}^{T}\vartheta_{1}(t)\|f(x)\|_{L^{2}(D)}^{2}dt+C\int_{0}^{T}\vartheta_{2}(t)\|g(x)\|_{L^{2}(D)}^{2}dt.
		}
	\end{equation*}
	By the resolvent estimate \eref{eqresoest}, we have $\|\tilde{E}(z)\|\leq C|z|^{-1}$, which leads to $\|E(t)\|\leq C$. Thus
	\begin{equation*}
		\vartheta_{1}\leq C\left(\int_{0}^{t}|h(r)|dr\right)^{2}\leq C\|h\|_{L^{\infty}([0,T])}^{2}.
	\end{equation*}
	As for $\vartheta_{2}$, using Lemma \ref{eqcorleml}, we have
	\begin{equation*}
		\eqalign{
			\vartheta_{2}\leq\int_{0}^{t}\left \|{}_{0}\partial^{\frac{1-2H}{2}}_{r}E(r)\right \|^{2}dr.
		}
	\end{equation*}
	The definition of $E$ and the resolvent estimate \eref{eqresoest} give
	\begin{equation*}
		\eqalign{
			\vartheta_{2}&\leq C\int_{0}^{t}\left (\int_{\Gamma_{\theta,\kappa}}|e^{zr}||z|^{\alpha-1}\|(z^{\alpha}+A^{s})^{-1}\||z|^{\frac{1-2H}{2}}|dz|\right )^{2}dr\\
			&\leq C\int_{0}^{t}\left (\int_{\Gamma_{\theta,\kappa}}|e^{zr}||z|^{-\frac{1}{2}-H}|dz|\right )^{2}dr\\
			&\leq C\int_{0}^{t}r^{2H-1}dr\\
			&\leq Ct^{2H}.
		}
	\end{equation*}
	As for $\mathbb{E}\|A^{\sigma s}u\|^{2}_{L^{2}(D)}$, we can get
	\begin{equation*}
		\eqalign{
			\mathbb{E}\|A^{\sigma s}u\|^{2}_{L^{2}(D)}&\leq C\left \|\int_{0}^{t}A^{\sigma s}E(t-r)h(r)f(x)dr\right \|_{L^{2}(D)}^{2}\\
			&\quad+C\mathbb{E}\left \|\int_{0}^{t}A^{\sigma s}E(t-r)g(x)dW^{H}(r)\right \|_{L^{2}(D)}^{2}\\
			&\leq\vartheta_{3}+\vartheta_{4}.
		}
	\end{equation*}
	The resolvent estimate gives
	\begin{equation*}
		 \fl\eqalign{
			\vartheta_{3}&\leq C\Bigg(\int_{0}^{t}\left\|\int_{\Gamma_{\theta,\kappa}}e^{z(t-r)}A^{\sigma s}z^{\alpha-1}(z^{\alpha}+A^{s})^{-1}dz\right\||h(r)|dr\Bigg)^{2}\|f(x)\|_{L^{2}(D)}^{2}\\
			&\leq C\Bigg(\int_{0}^{t}\int_{\Gamma_{\theta,\kappa}}|e^{z(t-r)}|\left\|A^{\sigma s}z^{\alpha-1}(z^{\alpha}+A^{s})^{-1}\right\||dz||h(r)|dr\Bigg)^{2}\|f(x)\|_{L^{2}(D)}^{2}\\
			&\leq C\Bigg(\int_{0}^{t}\int_{\Gamma_{\theta,\kappa}}|e^{z(t-r)}||z|^{\sigma\alpha-1}|dz||h(r)|dr\Bigg)^{2}\|f(x)\|_{L^{2}(D)}^{2}\\
			&\leq C\Bigg(\int_{0}^{t}(t-r)^{-\sigma \alpha}|h(r)|dr\Bigg)^{2}\|f(x)\|_{L^{2}(D)}^{2}\\
			&\leq C\|h\|_{L^{\infty}([0,T])}^{2}\|f(x)\|_{L^{2}(D)}^{2},
		}
	\end{equation*}
	where $\sigma\in[0,1]$. As for $\vartheta_{4}$, with the help of Lemma \ref{eqcorleml} and resolvent \eref{eqresoest}, we obtain
	\begin{equation*}
		\eqalign{
			\vartheta_{4}&\leq C\int_{0}^{t}\left\|\int_{\Gamma_{\theta,\kappa}}e^{zr}z^{\frac{1-2H}{2}}A^{\sigma s}z^{\alpha-1}(z^{\alpha}+A^{s})^{-1}dz\right \|^{2}dr\|g(x)\|_{L^{2}(D)}\\
			&\leq C\int_{0}^{t}\left(\int_{\Gamma_{\theta,\kappa}}|e^{zr}|\left\|z^{\frac{1-2H}{2}}A^{\sigma s}z^{\alpha-1}(z^{\alpha}+A^{s})^{-1}\right \||dz|\right)^{2}dr\|g(x)\|_{L^{2}(D)}\\
			&\leq C\int_{0}^{t}\left(\int_{\Gamma_{\theta,\kappa}}|e^{zr}||z|^{\frac{1-2H}{2}+\sigma\alpha-1}|dz|\right)^{2}dr\|g(x)\|_{L^{2}(D)}\\
			&\leq C\int_{0}^{t}r^{2H-1-2\sigma\alpha}dr\|g(x)\|_{L^{2}(D)},
		}
	\end{equation*}
	where we need to require $2H-2\sigma\alpha>0$ to preserve the boundedness of $\vartheta_{4}$, i.e., $\sigma\leq \min(1,\frac{H}{\alpha}-\epsilon)$ with $\epsilon>0$ arbitrary small.
\end{proof}
\section{The inverse problem}\label{Sec4}
Now, we first provide the reconstructions of $f$ and $|g|$, and then show the uniqueness and instability of the reconstructions.
\subsection{Reconstruction of $f$ and $|g|$}
Introduce $u_{k}(t)=(u(t),\phi_{k})$ and denote $(\cdot,\cdot)$ as $L^{2}$ inner product. Using Eq. \eref{eqsolrep1}, one can get
\begin{equation*}\fl
	u_{k}(t)=\int_{0}^{t}E_{\alpha,1}(-\lambda_{k}^{s}(t-r)^{\alpha})h(r)drf_{k}+\int_{0}^{t}E_{\alpha,1}(-\lambda_{k}^{s}(t-r)^{\alpha})dW^{H}(r)g_{k},
\end{equation*}
where $f_{k}=(f,\phi_{k})$ and $g_{k}=(g,\phi_{k})$. Thus we have
\begin{equation}\label{equdeterf}
	\mathbb{E}(u_{k}(t))=\int_{0}^{t}E_{\alpha,1}(-\lambda_{k}^{s}(t-r)^{\alpha})h(r)drf_{k}
\end{equation}
and the covariance between $u_{k}$ and $u_{l}$ is
\begin{equation}\label{equdeterg}\fl
	{\rm Cov}(u_{k}(t),u_{l}(t))=g_{k}g_{l}\mathbb{E}\Bigg(\int_{0}^{t}E_{\alpha,1}(-\lambda_{k}^{s}(t-r)^{\alpha})dW^{H}(r)\int_{0}^{t}E_{\alpha,1}(-\lambda_{l}^{s}(t-r)^{\alpha})dW^{H}(r)\Bigg).
\end{equation}
Then we provide some lemmas, which play key roles in the discussion of the uniqueness.
\begin{lemma}\label{lemunq1}
	Let $h(t)\geq C_{h}>0$ for $t\in(0,T)$. Then for each fixed $k\in\mathbb{N}^{*}$, there exists a positive constant $C_{1}$ such that
	\begin{equation*}
		\int_{0}^{t}E_{\alpha,1}(-\lambda_{k}^{s}(t-r)^{\alpha})h(r)dr\geq C_{1}>0.
	\end{equation*}
\end{lemma}
\begin{proof}
	According to Lemma \ref{lemmtal} and the property of $h(t)$, one can get
	\begin{equation*}
		\eqalign{
			&\int_{0}^{t}E_{\alpha,1}(-\lambda_{k}^{s}(t-r)^{\alpha})h(r)dr\\
			\geq &CE_{\alpha,1}(-\lambda_{k}^{s}t^{\alpha})\int_{0}^{t}h(r)dr\\
			\geq& C_{1}>0.
		}
	\end{equation*}
\end{proof}

\begin{lemma}\label{lemunq2}
	For each fixed $k,~l\in\mathbb{N}^{*}$, there exists a positive constant $C_{2}$ such that
	\begin{equation*}\fl
		\mathbb{E}\Bigg(\int_{0}^{t}E_{\alpha,1}(-\lambda_{k}^{s}(t-r)^{\alpha})dW^{H}(r)\int_{0}^{t}E_{\alpha,1}(-\lambda_{l}^{s}(t-r)^{\alpha})dW^{H}(r)\Bigg)\geq C_{2}>0.
	\end{equation*}
\end{lemma}
\begin{proof}
	For $H=\frac{1}{2}$, from the It\'o isometry one can obtain 
	\begin{equation*}
		\eqalign{
			&\mathbb{E}\Bigg(\int_{0}^{t}E_{\alpha,1}(-\lambda_{k}^{s}(t-r)^{\alpha})dW^{H}(r)\int_{0}^{t}E_{\alpha,1}(-\lambda_{l}^{s}(t-r)^{\alpha})dW^{H}(r)\Bigg)\\
			=&\int_{0}^{t}E_{\alpha,1}(-\lambda_{k}^{s}(t-r)^{\alpha})E_{\alpha,1}(-\lambda_{l}^{s}(t-r)^{\alpha})dr\\
			\geq&CE_{\alpha,1}(-\lambda_{k}^{s}t^{\alpha})E_{\alpha,1}(-\lambda_{l}^{s}t^{\alpha})\geq C> 0.
		}
	\end{equation*}
	For $H\in(\frac{1}{2},1)$, Lemma \ref{Lemitoeqge05} shows
	\begin{equation*}
		\eqalign{
			&\mathbb{E}\Bigg(\int_{0}^{t}E_{\alpha,1}(-\lambda_{k}^{s}(t-r)^{\alpha})dW^{H}(r)\int_{0}^{t}E_{\alpha,1}(-\lambda_{l}^{s}(t-r)^{\alpha})dW^{H}(r)\Bigg)\\
			=&H(2H-1)\int_{0}^{t}\int_{0}^{t}E_{\alpha,1}(-\lambda_{k}^{s}(t-r_{1})^{\alpha})E_{\alpha,1}(-\lambda_{l}^{s}(t-r_{2})^{\alpha})|r_{1}-r_{2}|^{2H-2}dr_{1}dr_{2}\\
			\geq&CE_{\alpha,1}(-\lambda_{k}^{s}t^{\alpha})E_{\alpha,1}(-\lambda_{l}^{s}t^{\alpha})\int_{0}^{t}\int_{0}^{t}|r_{1}-r_{2}|^{2H-2}dr_{1}dr_{2}\\
			\geq&C>0.
		}
	\end{equation*}
	As for $H\in(0,\frac{1}{2})$, introducing
	\begin{equation*}
		E_{k}(r)=\left\{\eqalign{
			&E_{\alpha,1}(-\lambda_{k}^{s}(t-r)^{\alpha}),\qquad r\in [0,t];\\
			&0,\qquad r\in[0,t]^{C},
		}\right.
	\end{equation*}
	one can obtain
	\begin{equation*}
		\eqalign{
			&\mathbb{E}\Bigg(\int_{0}^{t}E_{\alpha,1}(-\lambda_{k}^{s}(t-r)^{\alpha})dW^{H}(r)\int_{0}^{t}E_{\alpha,1}(-\lambda_{l}^{s}(t-r)^{\alpha})dW^{H}(r)\Bigg)\\
			=&\frac{1}{2}H(1-2H)\int_{\mathbb{R}}\int_{\mathbb{R}}\frac{(E_{k}(r_{1})-E_{k}(r_{2}))(E_{l}(r_{1})-E_{l}(r_{2}))}{|r_{1}-r_{2}|^{2-2H}}dr_{1}dr_{2}\\
			=&\frac{1}{2}H(1-2H)\int\int_{[0,t]\times[0,t]}\frac{(E_{k}(r_{1})-E_{k}(r_{2}))(E_{l}(r_{1})-E_{l}(r_{2}))}{|r_{1}-r_{2}|^{2-2H}}dr_{1}dr_{2}\\
			&+H(1-2H)\int\int_{[0,t]\times[0,t]^{C}}\frac{E_{k}(r_{1})E_{l}(r_{1})}{|r_{1}-r_{2}|^{2-2H}}dr_{1}dr_{2}.
		}
	\end{equation*}
	Combining Lemma \ref{lemmtal}, we have
	\begin{equation*}
		\frac{1}{2}H(1-2H)\int\int_{[0,t]\times[0,t]}\frac{(E_{k}(r_{1})-E_{k}(r_{2}))(E_{l}(r_{1})-E_{l}(r_{2}))}{|r_{1}-r_{2}|^{2-2H}}dr_{1}dr_{2}\geq 0,
	\end{equation*}
	and
	\begin{equation*}
		\eqalign{
			&\mathbb{E}\Bigg(\int_{0}^{t}E_{\alpha,1}(-\lambda_{k}^{s}(t-r)^{\alpha})dW^{H}(r)\int_{0}^{t}E_{\alpha,1}(-\lambda_{l}^{s}(t-r)^{\alpha})dW^{H}(r)\Bigg)\\
			\geq&H(1-2H)E_{k}(t)E_{l}(t)\int\int_{[0,t]\times[0,t]^{C}}\frac{1}{|r_{1}-r_{2}|^{2-2H}}dr_{1}dr_{2}\\
			\geq&C_{2}>0.
		}
	\end{equation*}
\end{proof}

Combining Eqs. \eref{equdeterf}, \eref{equdeterg}, and Lemmas \ref{lemunq1}, \ref{lemunq2}, one can get the uniqueness of the inverse problem.

\begin{theorem}
	Let $f(x),g(x)\in L^{2}(D)$, and $\|g(x)\|_{L^{2}(D)}\neq0$. Then one can determine the source terms $f$ and $|g|$  uniquely by the data set
	\begin{equation*}
		\{(\mathbb{E}u_{k}(T),{\rm Cov}(u_{k}(T),u_{l}(T))):~k,~l\in\mathbb{N}^{*}\}.
	\end{equation*}
\end{theorem}
\begin{proof}
	According to \eref{equdeterf}, $f_{k}$ can be got by
	\begin{equation*}
		f_{k}=\frac{\mathbb{E}(u_{k})}{\int_{0}^{t}E_{\alpha,1}(-\lambda_{k}^{s}(t-r)^{\alpha})h(r)dr}.
	\end{equation*}
	Thus there exists $f=\sum\limits_{k=1}^{\infty}f_{k}\phi_{k}$. Further from \eref{equdeterg}, there is
	\begin{equation*}
		g_{k}g_{l}=\frac{{\rm Cov}(u_{k}(t),u_{l}(t))}{\mathbb{E}\left (\int_{0}^{t}E_{\alpha,1}(-\lambda_{k}^{s}(t-r)^{\alpha})dW^{H}(r)\int_{0}^{t}E_{\alpha,1}(-\lambda_{l}^{s}(t-r)^{\alpha})dW^{H}(r)\right )},
	\end{equation*}
	which can help us to obtain $g^{2}$.
\end{proof}
\subsection{Instability}
In this subsection, we show the instability of the inverse source problem. Different from the discussions in \cite{Feng.2020AirspfttfdedbafBm}, we show the instability of recovering $g_{k}g_{l}$. 

\begin{theorem}
	Let $f(x),g(x)\in L^{2}(D)$ and $\|g(x)\|_{L^{2}(D)}\neq0$. Then the following estimates hold
	\begin{equation*}
		\left |\int_{0}^{t}E_{\alpha,1}(-\lambda_{k}^{s}(t-r)^{\alpha})h(r)dr\right |\leq C\lambda_{k}^{-\sigma_{1} s}
	\end{equation*}
	and
	\begin{equation*}\fl
		\mathbb{E}\Bigg(\int_{0}^{t}E_{\alpha,1}(-\lambda_{k}^{s}(t-r)^{\alpha})dW^{H}(r)\int_{0}^{t}E_{\alpha,1}(-\lambda_{l}^{s}(t-r)^{\alpha})dW^{H}(r)\Bigg)\leq C\lambda_{k}^{-\sigma_{2} s}\lambda_{l}^{-\sigma_{2} s},
	\end{equation*}
	where $\sigma_{1}\in[0,1]$ and $\sigma_{2}\in[0,\min(\frac{H}{\alpha}-\epsilon,1)]$.
\end{theorem}
\begin{proof}
	The resolvent estimate and simple calculations imply
	\begin{equation*}
		\eqalign{
			&\left |\int_{0}^{t}E_{\alpha,1}(-\lambda_{k}^{s}(t-r)^{\alpha})h(r)dr\right |\\
			\leq&C\int_{0}^{t}\left |\int_{\Gamma_{\theta,\kappa}}e^{zr}z^{\alpha-1}(z^{\alpha}+\lambda_{k}^{s})^{-1}dz\right |dr\|h(r)\|_{L^{\infty}([0,T])}\\
			\leq&C\lambda_{k}^{-\sigma_{1} s}\int_{0}^{t}\left |\int_{\Gamma_{\theta,\kappa}}e^{zr}z^{\alpha-1}\lambda_{k}^{\sigma_{1} s}(z^{\alpha}+\lambda_{k}^{s})^{-1}dz\right |dr\|h(r)\|_{L^{\infty}([0,T])}\\
			\leq&C\lambda_{k}^{-\sigma_{1} s}\int_{0}^{t}\int_{\Gamma_{\theta,\kappa}}|e^{zr}||z|^{\sigma_{1}\alpha-1}|dz|dr\|h(r)\|_{L^{\infty}([0,T])}\\
			\leq&C\lambda_{k}^{-\sigma_{1} s}\int_{0}^{t}r^{-\sigma_{1}\alpha}dr,
		}
	\end{equation*}
	where we need to require $\sigma_{1}\in[0,1]$.
	According to Lemma \ref{eqcorleml} and the resolvent estimate, one can obtain
	\begin{equation*}
		\eqalign{
			&\mathbb{E}\Bigg(\int_{0}^{t}E_{\alpha,1}(-\lambda_{k}^{s}(t-r)^{\alpha})dW^{H}(r)\int_{0}^{t}E_{\alpha,1}(-\lambda_{l}^{s}(t-r)^{\alpha})dW^{H}(r)\Bigg)\\
			\leq&C\left(\int_{0}^{t}\left |{}_{0}\partial^{\frac{1-2H}{2}}_{r}E_{\alpha,1}(-\lambda_{k}^{s}r^{\alpha})\right |^{2}dr\right)^{\frac{1}{2}}\left(\int_{0}^{t}\left |{}_{0}\partial^{\frac{1-2H}{2}}_{r}E_{\alpha,1}(-\lambda_{l}^{s}r^{\alpha})\right |^{2}dr\right)^{\frac{1}{2}}\\
			\leq&C\left(\int_{0}^{t}\left |\int_{\Gamma_{\theta,\kappa}}e^{zr}z^{\frac{1-2H}{2}}z^{\alpha-1}(z^{\alpha}+\lambda_{k}^{s})^{-1}dz\right |^{2}dr\right)^{\frac{1}{2}}\\
			&\qquad \cdot\left(\int_{0}^{t}\left |\int_{\Gamma_{\theta,\kappa}}e^{zr}z^{\frac{1-2H}{2}}z^{\alpha-1}(z^{\alpha}+\lambda_{l}^{s})^{-1}dz\right |^{2}dr\right)^{\frac{1}{2}}\\
			\leq&C\lambda_{k}^{-\sigma_{2} s}\lambda_{l}^{-\sigma_{2} s}\left(\int_{0}^{t}\left |\int_{\Gamma_{\theta,\kappa}}e^{zr}z^{\frac{1-2H}{2}}z^{\alpha-1}\lambda_{k}^{-\sigma_{2} s}(z^{\alpha}+\lambda_{k}^{s})^{-1}dz\right |^{2}dr\right)^{\frac{1}{2}}\\
			&\qquad\cdot\left(\int_{0}^{t}\left |\int_{\Gamma_{\theta,\kappa}}e^{zr}z^{\frac{1-2H}{2}}z^{\alpha-1}\lambda_{l}^{-\sigma_{2} s}(z^{\alpha}+\lambda_{l}^{s})^{-1}dz\right |^{2}dr\right)^{\frac{1}{2}}\\
			\leq&C\lambda_{k}^{-\sigma_{2} s}\lambda_{l}^{-\sigma_{2} s}\int_{0}^{t}\left (\int_{\Gamma_{\theta,\kappa}}|e^{zr}||z|^{\frac{1-2H}{2}+\sigma_{2}\alpha-1}|dz|\right )^{2}dr\\
			\leq&C\lambda_{k}^{-\sigma_{2} s}\lambda_{l}^{-\sigma_{2} s}\int_{0}^{t}r^{2H-1-2\sigma_{2}\alpha}dr,
		}
	\end{equation*}
	where we need to require $2H-2\sigma_{2}\alpha>0$, i.e., $\sigma_{2}<\frac{H}{\alpha}$.
\end{proof}
\section{Numerical experiments} \label{Sec5}
In this section, we provide some examples to validate the developed theory. Here, we choose $D=(0,1)$. Thus the eigenvalues and eigenfunctions of $-\Delta$ are
\begin{equation*}
	\lambda_{k}=k^{2}\pi^{2},\quad \phi_{k}(x)=\sqrt{2}\sin(k\pi x),\quad k=1,2,\cdots.
\end{equation*}

\subsection{Numerical scheme for direct problem}
To get synthetic data for inverse problem,  we propose a numerical scheme to solve direct problem numerically.

Introduce $\mathbb{H}_{N}={\rm span}\{\phi_{1},\phi_{2},\ldots,\phi_{N}\}$ with $N\in\mathbb{N}^{*}$ and define the projection operator $P_{N}$: $L^{2}(\Omega)\rightarrow \mathbb{H}_{N}$ by
\begin{equation*}
	P_{N}u=\sum_{i=1}^{N}(u,\phi_{i})\phi_{i}\quad \forall u\in L^{2}(\Omega).
\end{equation*}
Define $A^{s}_{N}:~\mathbb{H}_{N}\rightarrow \mathbb{H}_{N}$ as
\begin{equation*}
	A^{s}_{N}v_{N}=\sum_{k=1}^{N}\lambda_{k}^{s}(v_{N},\phi_{k})\phi_{k}\quad v_{N}\in\mathbb{H}_{N}.
\end{equation*}
Let $\tau=T/L$ with $L\in\mathbb{N}^{*}$ and $t_{k}=k\tau$ ($k=0,1,2,\cdots,L$). Then the numerical solution of $u(x,t_{n})$, i.e., $u^{n}_{N}$, can be approximated by
\begin{equation*}\fl
	\frac{u^{n}_{N}-u^{n-1}_{N}}{\tau}+\sum_{i=0}^{n-1}d^{(1-\alpha)}_{i}A^{s}_{N}u^{n-i}_{N}=P_{N}f(x)h(t_{n})+P_{N}g(x)\frac{W^{H}(t_{n})-W^{H}(t_{n-1})}{\tau},
\end{equation*}
where
\begin{equation*}
	\sum_{i=0}^{\infty}d^{(\alpha)}_{i}\xi^{i}=\left(\frac{1-\xi}{\tau}\right)^{\alpha}.
\end{equation*}
Further introduce $v_{1,k}(t)$ and $v_{2,k}(t)$ as
\begin{equation*}
	\eqalign{
		&v_{1,k}(t)=\int_{0}^{t}E_{\alpha,1}(-\lambda_{k}^{s}(t-r)^{\alpha})h(r)dr,\\
		&v_{2,k}(t)=\int_{0}^{t}E_{\alpha,1}(-\lambda_{k}^{s}(t-r)^{\alpha})dW^{H}(r).
	}
\end{equation*}
It is easy to verify that $v_{1,k}(t)$ and $v_{2,k}(t)$ are the solutions of
\begin{equation*}
	\left\{
	\eqalign{
		&\partial_{t}v_{1,k}(t)+\partial^{1-\alpha}_{t} \lambda_{k}^{s}v_{1,k}(t)=h(t)\qquad t\in (0,T],\\
		&v_{1,k}(0)=0
	}
	\right.
\end{equation*}
and
\begin{equation*}
	\left\{
	\eqalign{
		&\partial_{t}v_{2,k}(t)+\partial^{1-\alpha}_{t} \lambda_{k}^{s}v_{2,k}(t)=\dot{W}^{H}(t)\qquad t\in (0,T],\\
		&v_{2,k}(0)=0.
	}
	\right.
\end{equation*}
The  $v_{1,k}(t_{n})$ and $v_{2,k}(t_{n})$ can be respectively approximated by $v_{1,k}^{n}$ and $v_{2,k}^{n}$, which are the solutions of
\begin{equation*}
	\frac{v_{1,k}^{n}-v_{1,k}^{n-1}}{\tau}+\sum_{i=0}^{n-1}d^{(1-\alpha)}_{i}\lambda_{k}^{s}v_{1,k}^{n-i}=h(t_{n})
\end{equation*}
and
\begin{equation*}
	\frac{v_{2,k}^{n}-v_{2,k}^{n-1}}{\tau}+\sum_{i=0}^{n-1}d^{(1-\alpha)}_{i}\lambda_{k}^{s}v_{2,k}^{n-i}=\frac{W^{H}(t_{n})-W^{H}(t_{n-1})}{\tau}.
\end{equation*}
\subsection{Numerical results}
In the numerical experiments, we take $T=0.5$, $N=20$, and $L=1024$. As for $f(x)$, $g(x)$, and $h(t)$, we choose
\begin{equation*}
	h(t)=t+1,\qquad f(x)=4x(1-x)(1-2x),\qquad g(x)=x(1-x)^{2}.
\end{equation*}
We generate $1000$  trajectories to approximate the $\mathbb{E}(u_{k})$ and ${\rm Cov}(u_{k},u_{l})$, which can help us to recover $f$ and $|g|$.

Here, we take $(\alpha,s,H)$ as $\{\alpha=0.4,s=0.3,H=0.2\}$, $\{\alpha=0.6,s=0.7,H=0.2\}$, $\{\alpha=0.6,s=0.3,H=0.5\}$,
$\{\alpha=0.4,s=0.7,H=0.5\}$,
$\{\alpha=0.3,s=0.4,H=0.8\}$, and
$\{\alpha=0.7,s=0.6,H=0.8\}$, respectively. The corresponding results are shown in Figures \ref{lab432}--\ref{lab768}. From these results, it can be seen that both $f$ and $|g|$ can be well recovered.
\begin{figure}
	\centering
	\subfigure[$f(x)$]{\includegraphics[width=0.45\linewidth,angle=0]{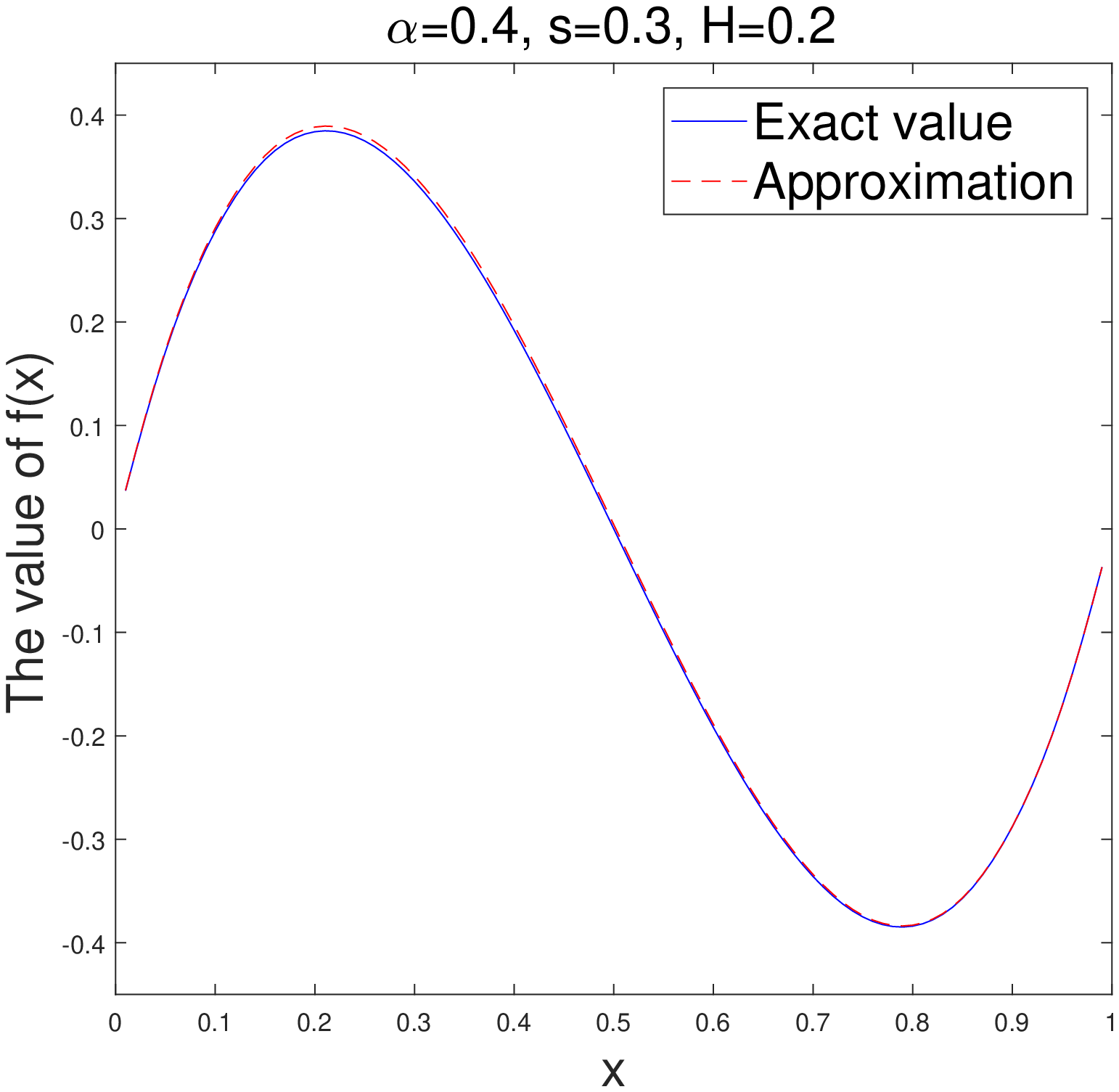}\label{lab:432f}}
	\subfigure[$|g(x)|$]{\includegraphics[width=0.45\linewidth,angle=0]{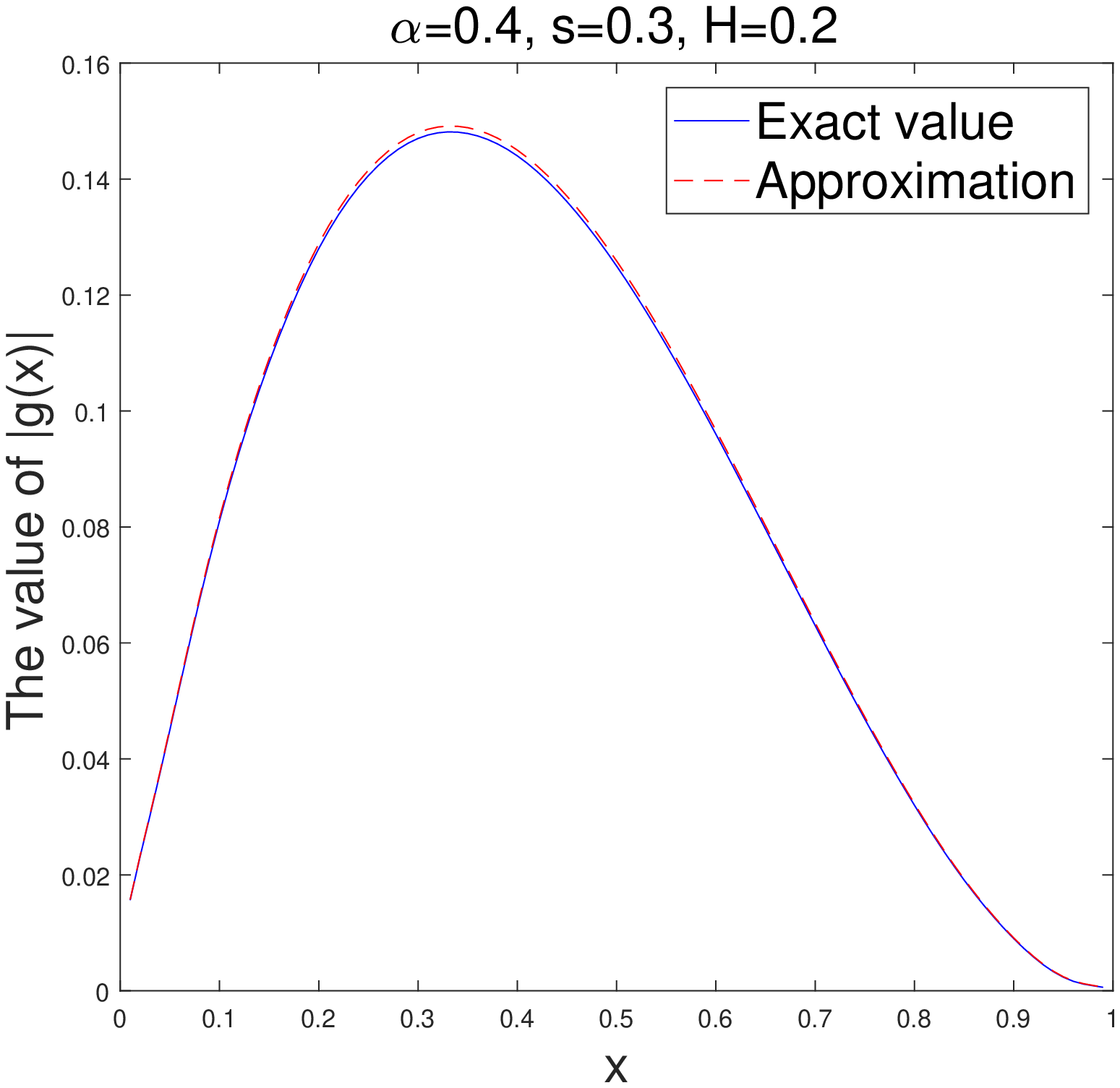}\label{lab:432g}}
	\caption{The exact and reconstructed solutions with $\alpha=0.4$, $s=0.3$, and $H=0.2$.}
	\label{lab432}
\end{figure}

\begin{figure}
	\centering
	\subfigure[$f(x)$]{\includegraphics[width=0.45\linewidth,angle=0]{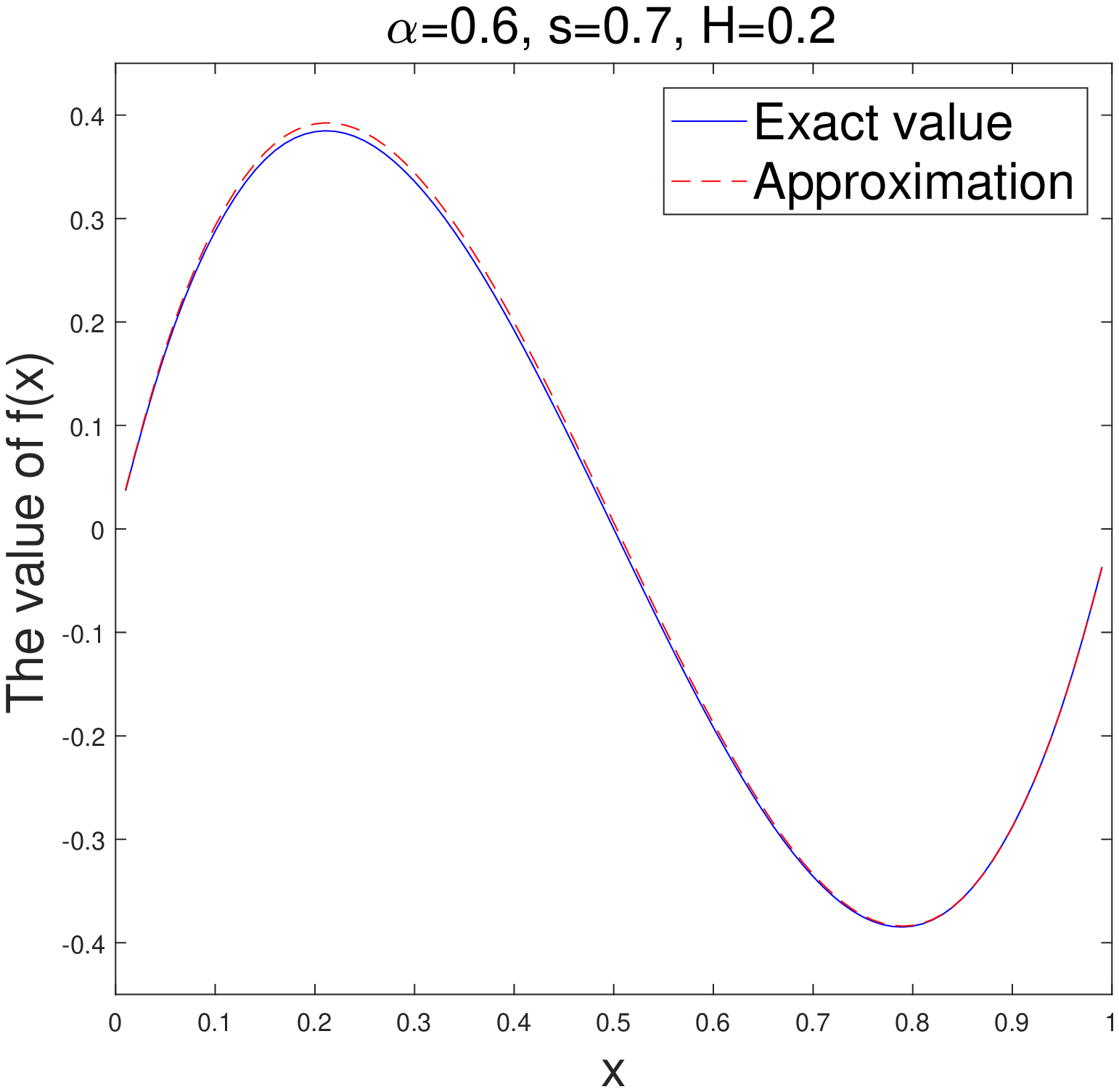}\label{lab:672f}}
	\subfigure[$|g(x)|$]{\includegraphics[width=0.45\linewidth,angle=0]{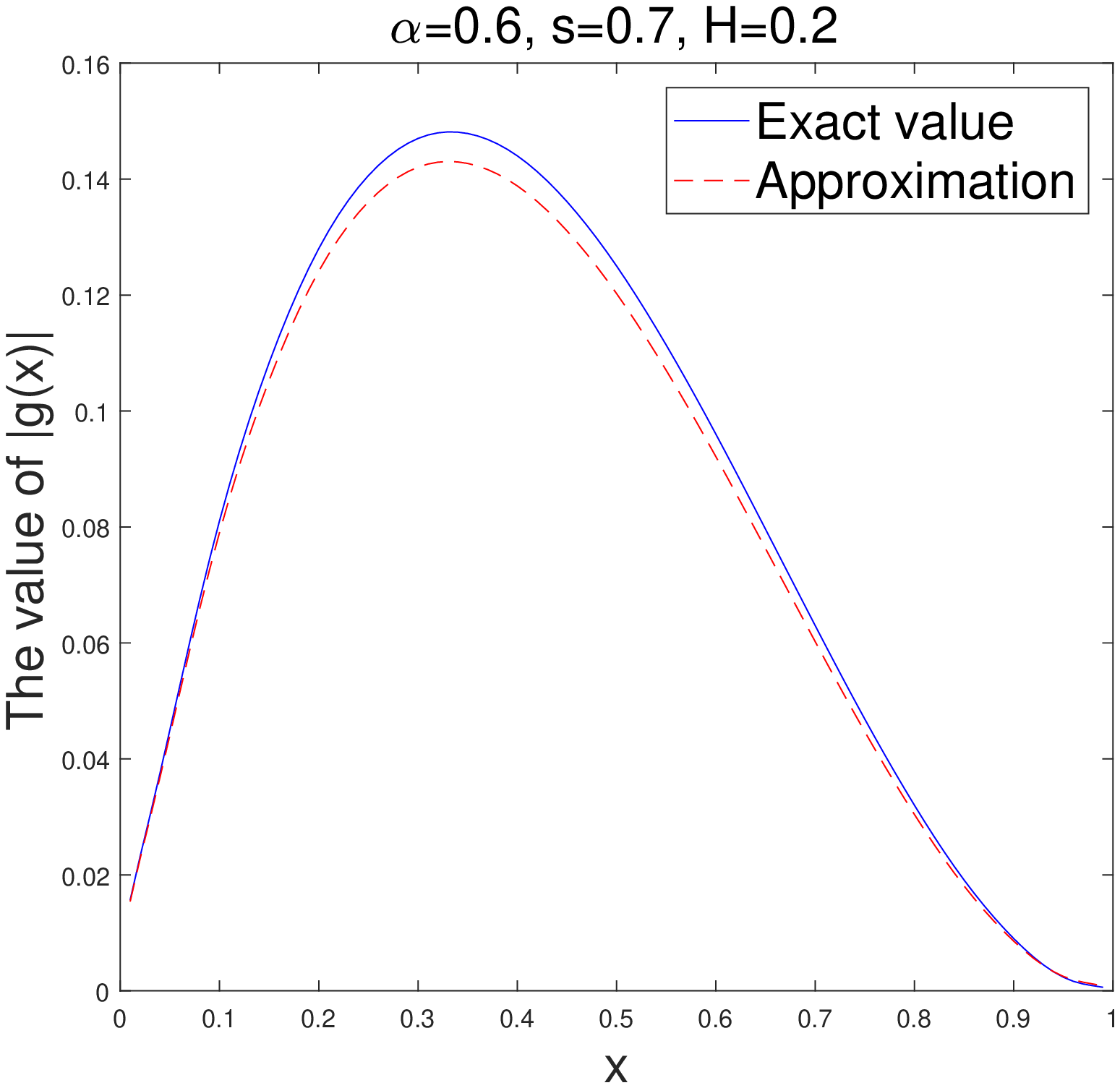}\label{lab:672g}}
	\caption{The exact and reconstructed solutions with $\alpha=0.6$, $s=0.7$ and $H=0.2$.}
	\label{lab672}
\end{figure}

\begin{figure}
	\centering
	\subfigure[$f(x)$]{\includegraphics[width=0.45\linewidth,angle=0]{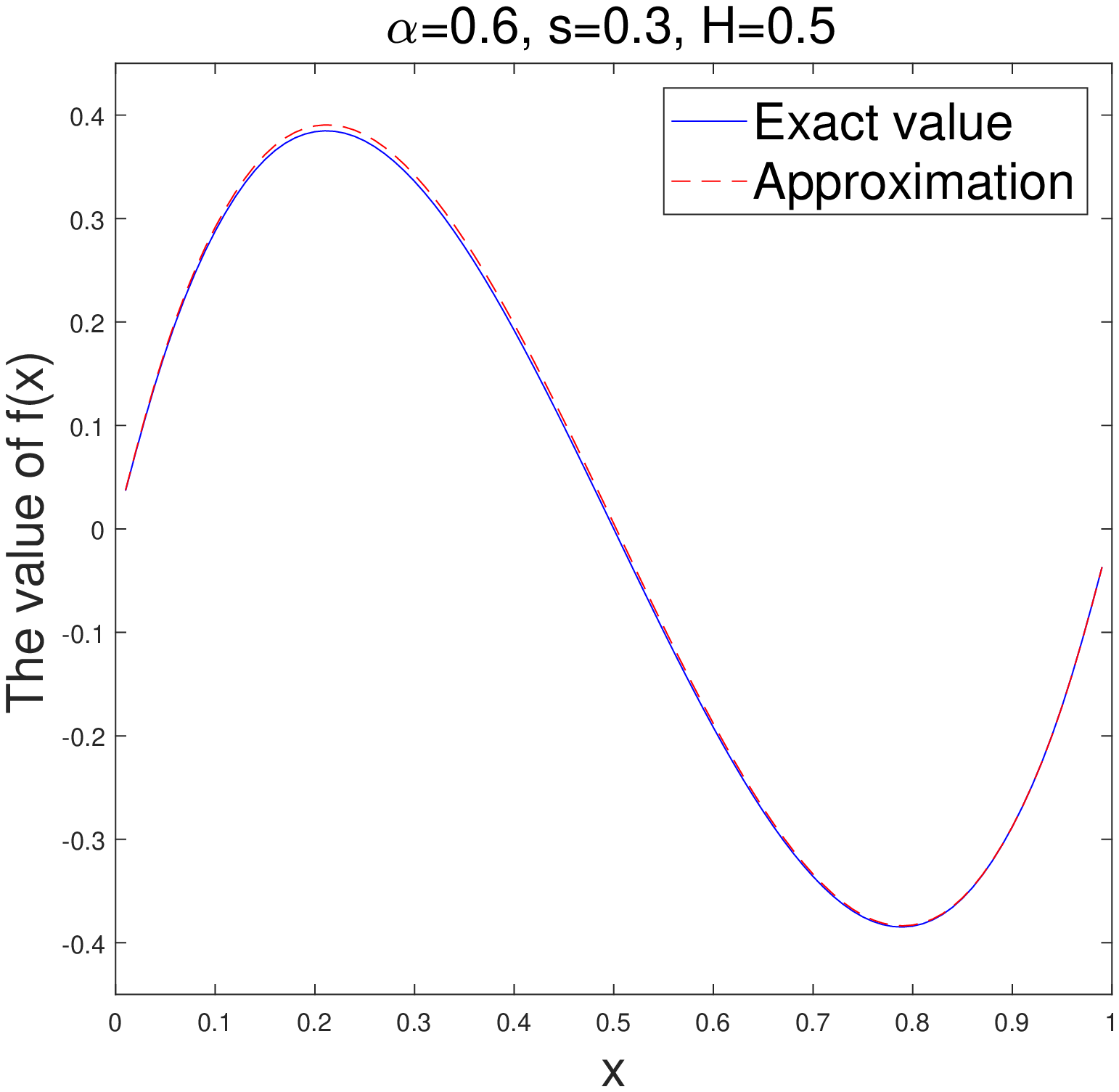}\label{lab:635f}}
	\subfigure[$|g(x)|$]{\includegraphics[width=0.45\linewidth,angle=0]{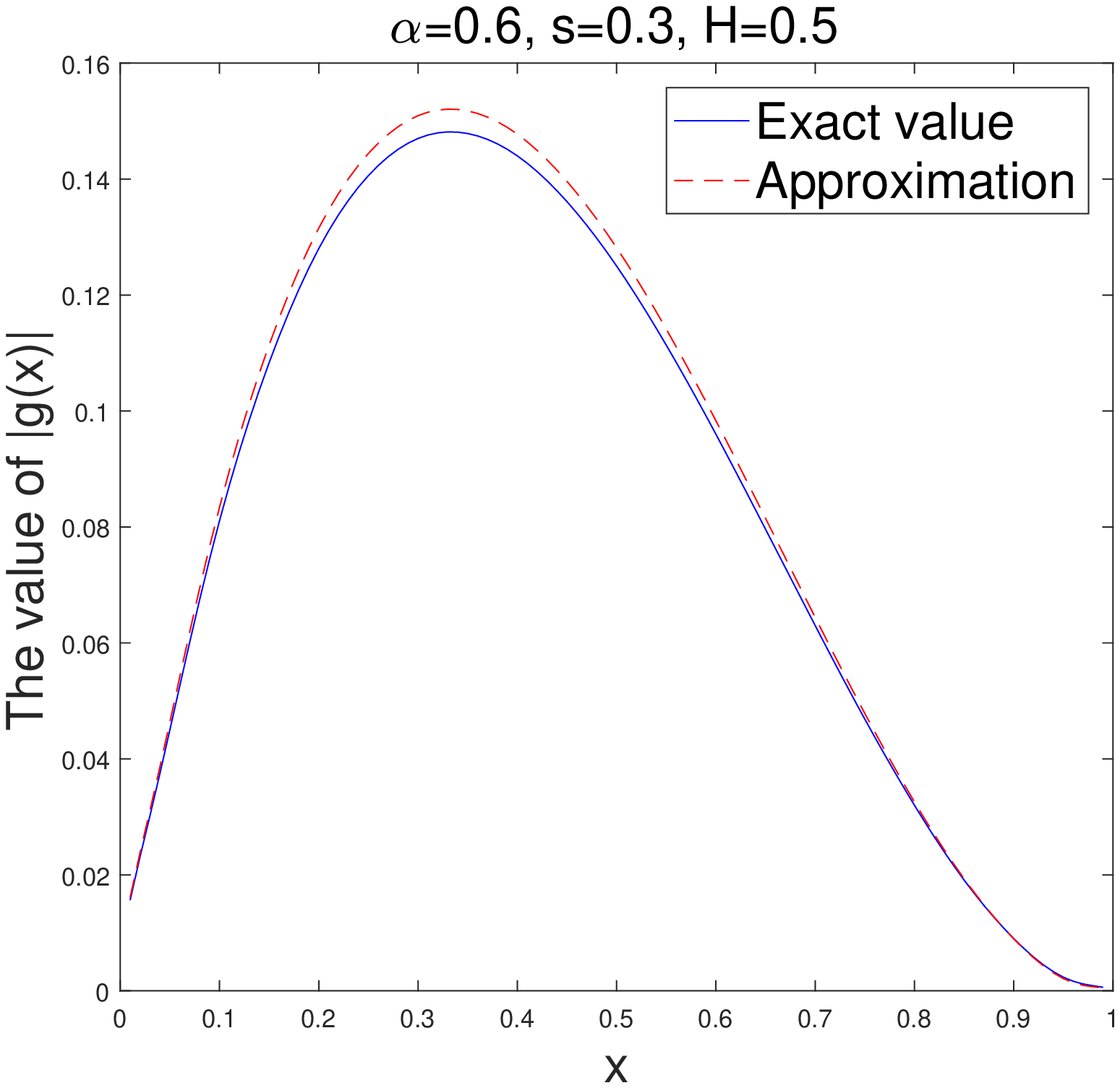}\label{lab:635g}}
	\caption{The exact and reconstructed solutions with $\alpha=0.6$, $s=0.3$, and $H=0.5$.}
	\label{lab635}
\end{figure}

\begin{figure}
	\centering
	\subfigure[$f(x)$]{\includegraphics[width=0.45\linewidth,angle=0]{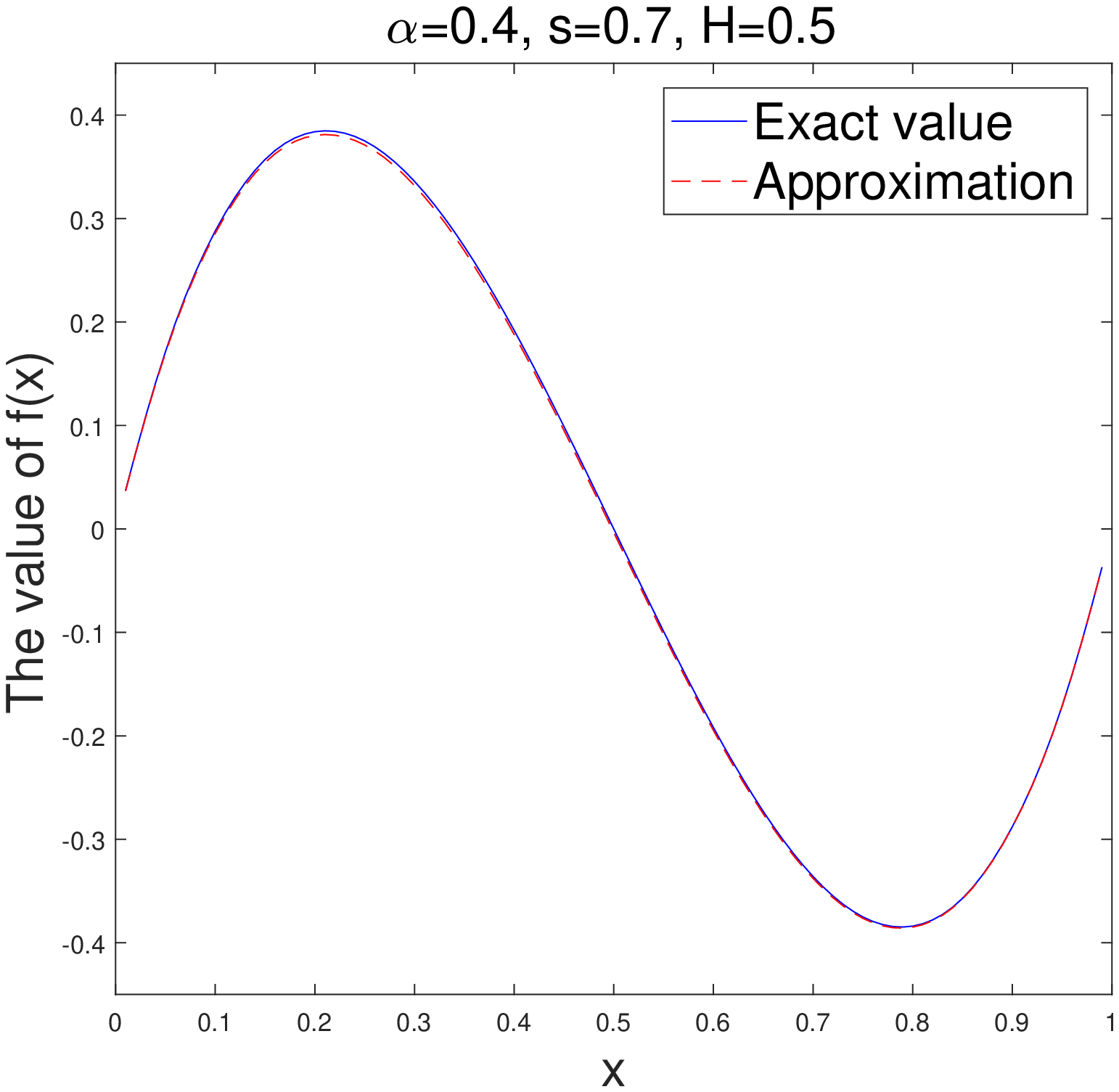}\label{lab:475f}}
	\subfigure[$|g(x)|$]{\includegraphics[width=0.45\linewidth,angle=0]{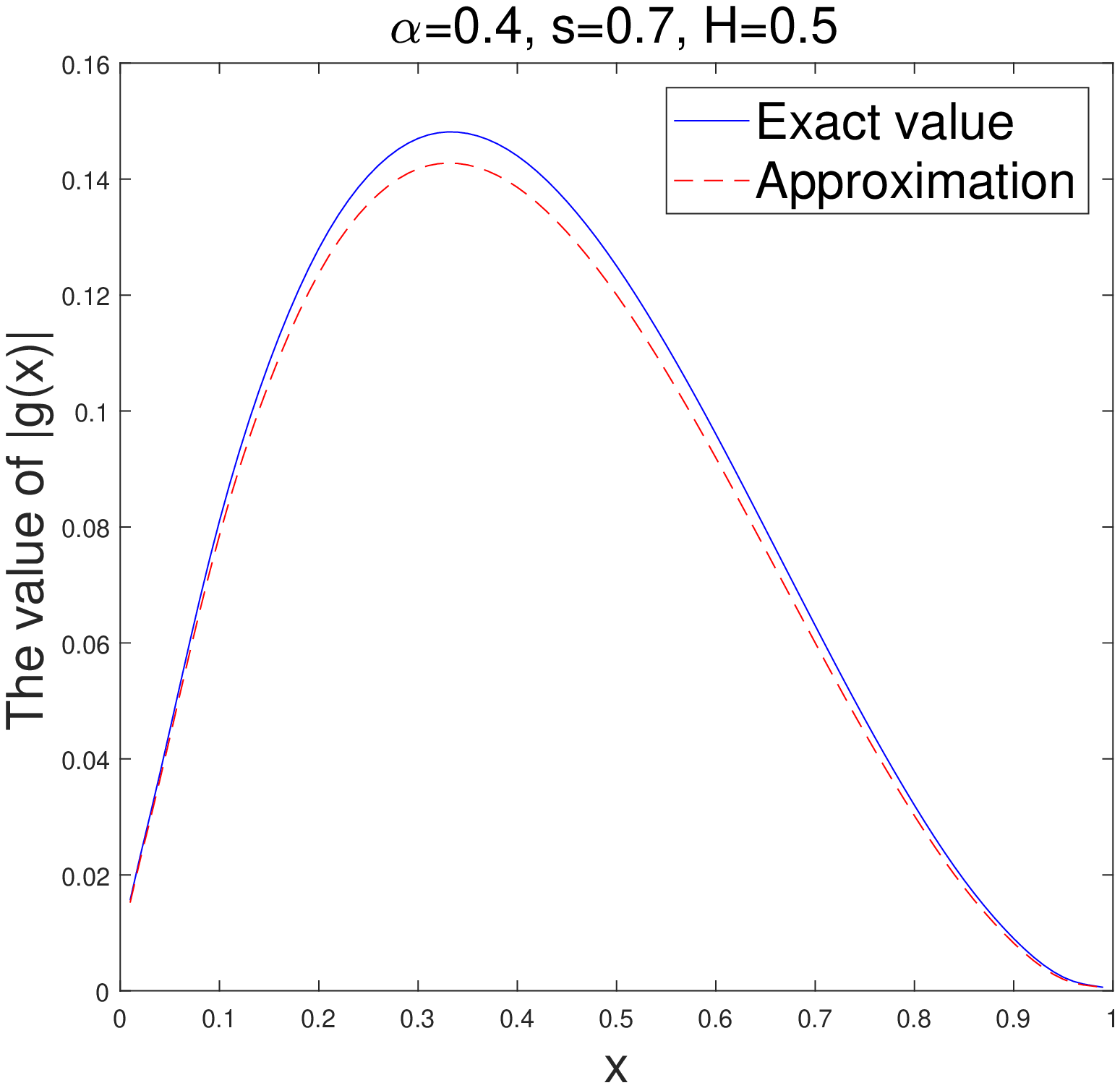}\label{lab:475g}}
	\caption{The exact and reconstructed solutions with $\alpha=0.4$, $s=0.7$, and $H=0.5$.}
	\label{lab475}
\end{figure}

\begin{figure}
	\centering
	\subfigure[$f(x)$]{\includegraphics[width=0.45\linewidth,angle=0]{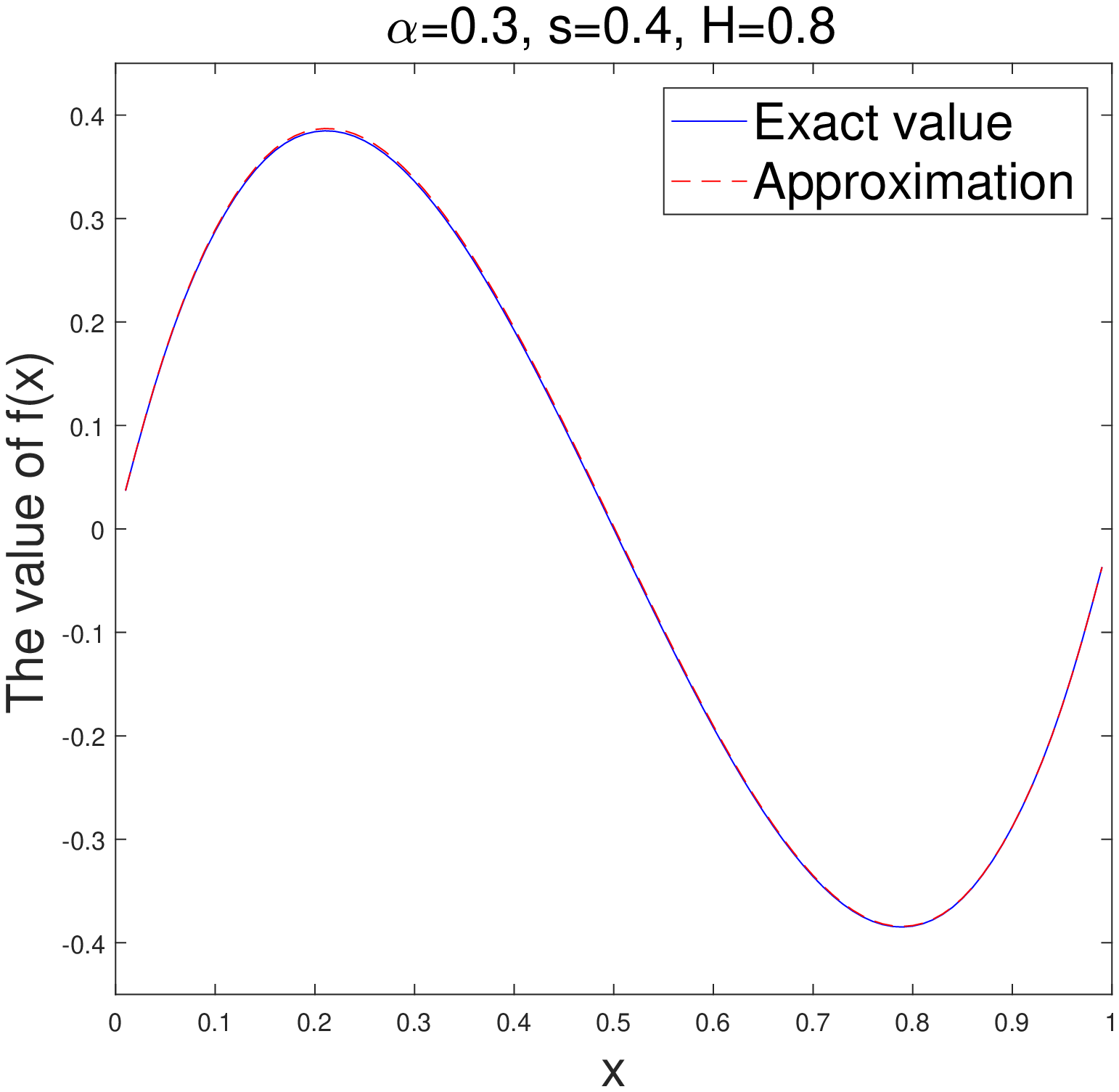}\label{lab:348f}}
	\subfigure[$|g(x)|$]{\includegraphics[width=0.45\linewidth,angle=0]{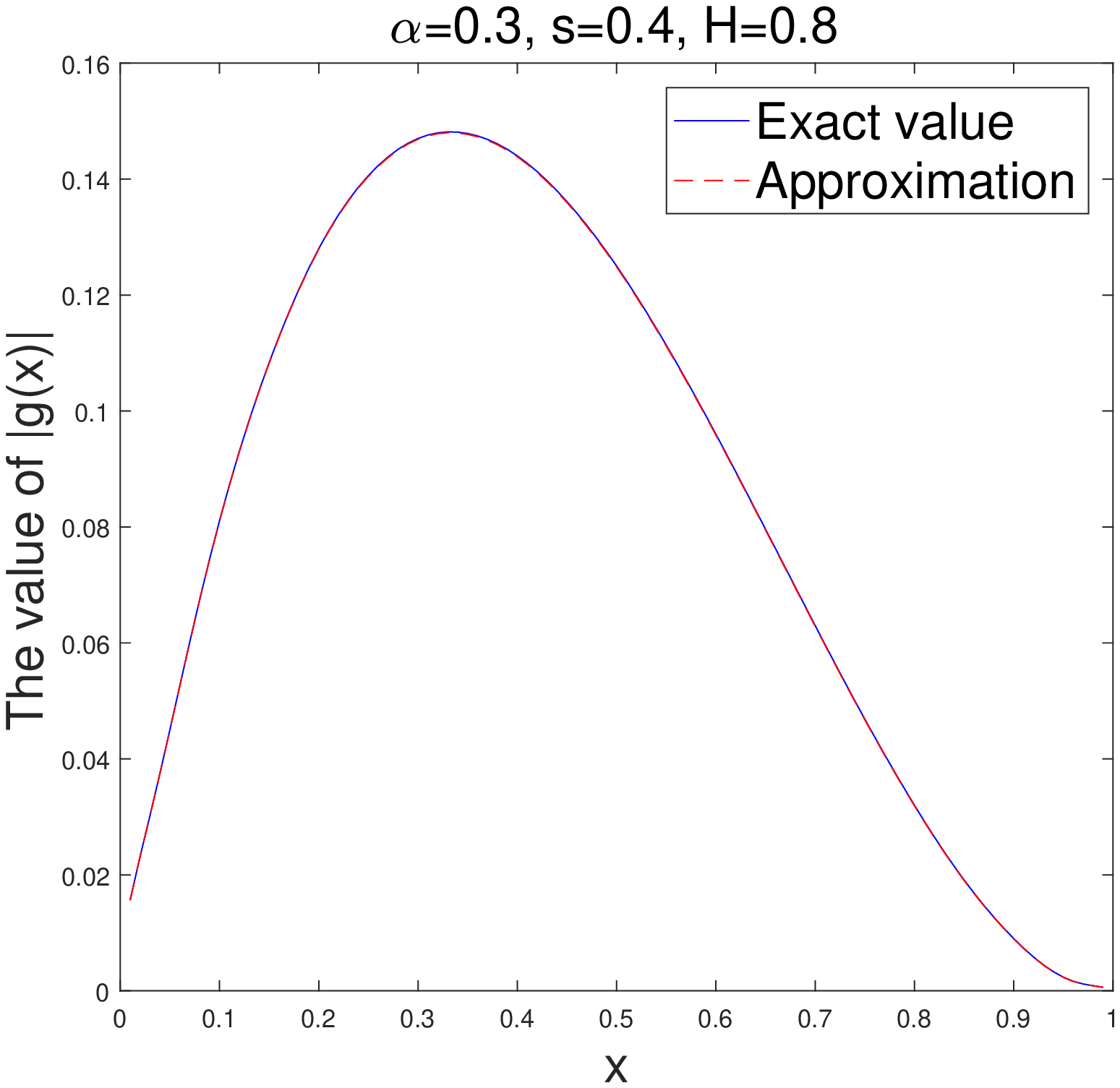}\label{lab:348g}}
	\caption{The exact and reconstructed solutions with $\alpha=0.3$, $s=0.4$, and $H=0.8$.}
	\label{lab348}
\end{figure}

\begin{figure}
	\centering
	\subfigure[$f(x)$]{\includegraphics[width=0.45\linewidth,angle=0]{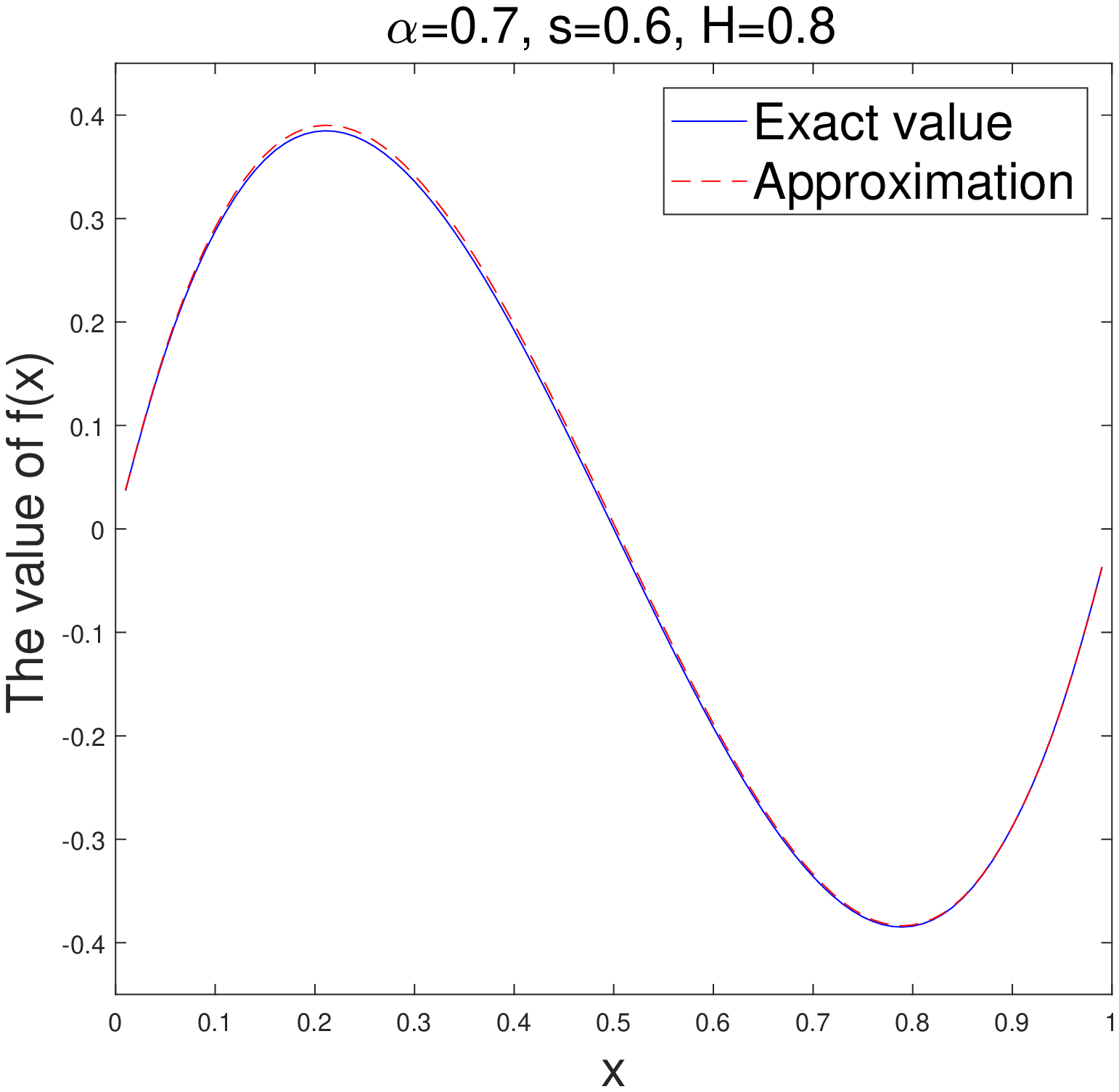}\label{lab:768f}}
	\subfigure[$|g(x)|$]{\includegraphics[width=0.45\linewidth,angle=0]{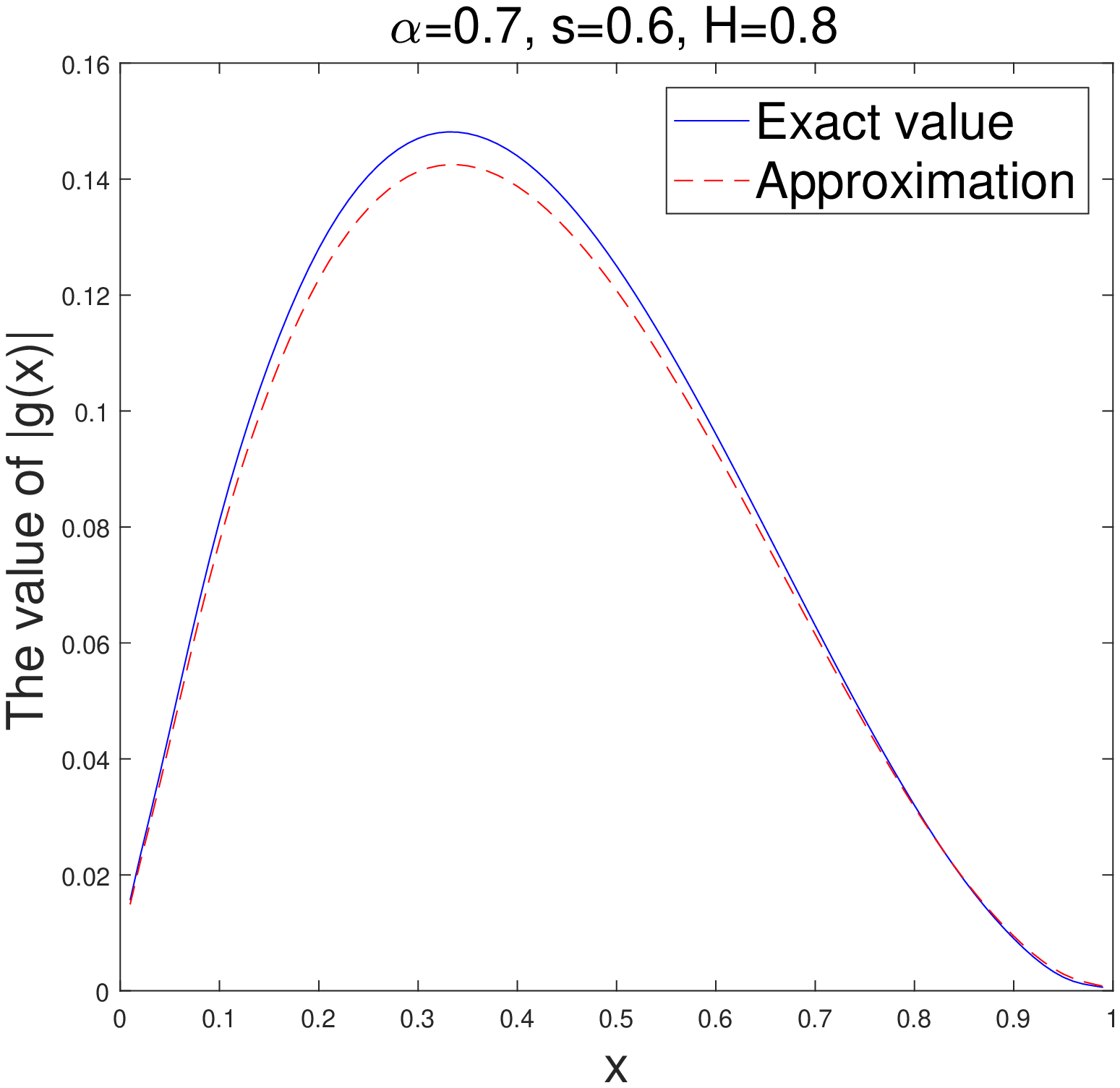}\label{lab:768g}}
	\caption{The exact and reconstructed solutions with $\alpha=0.7$, $s=0.6$, and $H=0.8$.}
	\label{lab768}
\end{figure}
\section{Conclusion}\label{Sec6}
We study the inverse random source problem for the time-space fractional diffusion equation driven by fBm with Hurst index $H\in(0,1)$. We first provide the well-posedness of the direct problem; and then show the uniqueness of the inverse problem. With the help of Lemma \ref{eqcorleml}, we give the instability of recovering $f$ and $|g|$. Numerical experiments validate the theoretical predictions.
\section*{Acknowledgment}
This work was supported by the National Natural Science Foundation of China under
Grant No. 12071195, and the AI and Big Data Funds under Grant No. 2019620005000775.
\section*{References}
\bibliographystyle{iopart-num2}

\bibliography{stoinv}

\end{document}